\documentclass[11pt]{amsart}  
\usepackage{color}
\usepackage{hyperref}

\def\R{\mathbb R}

\def\X{\mathcal X}

\def\pa{\partial}

\def\Om{\Omega}

\def\cal{\mathcal}

\def\div{{\rm div}\,}
\def\vol{{\rm vol}\,}

\def\<{\langle}
\def\>{\rangle}
\def\dist{{\rm dist}\,}

\def\cF{{\cal F}} 

\DeclareMathSymbol{\leqslant}{\mathalpha}{AMSa}{"36} 
\DeclareMathSymbol{\geqslant}{\mathalpha}{AMSa}{"3E} 
\DeclareMathSymbol{\eset}{\mathalpha}{AMSb}{"3F}     
\renewcommand{\leq}{\;\leqslant\;}                   
\renewcommand{\geq}{\;\geqslant\;}                   

\newcounter{mycount}

\newcommand{\cG}{\ensuremath{\mathcal G}}

\newcommand{\cL}{\ensuremath{\mathcal L}}

\newcommand{\gga}{\gamma}            
\newcommand{\gd}{\delta}
\newcommand{\gep}{\varepsilon}       
\newcommand{\gp}{\varphi}

\newcommand{\gl}{\lambda}
\newcommand{\gL}{\Lambda}
\newcommand{\gs}{\sigma}

\newcommand{\bbD}{{\ensuremath{\mathbb D}} }
\newcommand{\bbE}{{\ensuremath{\mathbb E}} }

\newcommand{\bbN}{{\ensuremath{\mathbb N}} }

\newcommand{\bbP}{{\ensuremath{\mathbb P}} }

\newcommand{\bbR}{{\ensuremath{\mathbb R}} }

\newcommand\dt{{\frac{\mathrm d}{\mathrm dt}}}

\newcommand{\dd}{{\, \mathrm d}}

\newcommand{\vs}{\vspace{0.3cm}}
\newcommand{\ds}{\displaystyle}

\newcommand{\sk}{\smallskip}
\newcommand{\mk}{\medskip}
\newcommand{\bk}{\bigskip}

\let\oldmarginpar\marginpar
\renewcommand\marginpar[1]{\-\oldmarginpar[\raggedleft\footnotesize #1]%
{\raggedright\footnotesize #1}}

\title[Boundary-driven nonlinear drift-diffusions]{Lyapunov
  functionals for boundary-driven nonlinear drift-diffusions}

\author{T. Bodineau, J.L. Lebowitz, C. Mouhot \& C.  Villani}

\def\signtb{\bigskip \begin{center} {\sc Thierry Bodineau
\par\vspace{3mm}
\'ENS Paris \& CNRS\par
DMA, UMR CNRS 8553\par
45 rue d'Ulm 
F 75320 Paris cedex 05 
FRANCE\par\vspace{3mm}
e-mail:} \tt{Thierry.Bodineau@ens.fr} \end{center}}

\def\signjl{\bigskip \begin{center} {\sc Joel Lebowitz\par\vspace{3mm}
Rutgers University\par
Departments of Mathematics and Physics\\par
110 Frelinghuysen Road
Piscataway, NJ 08854, USA\par\vspace{3mm}
e-mail:} \tt{lebowitz@math.rutgers.edu} \end{center}}

\def\signcv{\bigskip \begin{center} {\sc C\'edric
      Villani\par\vspace{3mm} Universit\'e Lyon 1, Institut Henri
      Poincar\'e \& IUF\par
      UMPA, UMR CNRS 5669\par
      46 all\'ee d'Italie 69364 Lyon Cedex 07 FRANCE\par\vspace{3mm}
      e-mail:} \tt{cvillani@umpa.ens-lyon.fr} \end{center}}

\def\signcm{\bigskip \begin{center} {\sc 
Cl\'ement Mouhot\par\vspace{3mm}
University of Cambridge\par
DPMMS, Centre for Mathematical Sciences\par
Wilberforce Road, 
Cambridge CB3 0WA, 
UK\par\vspace{3mm}
e-mail:} \tt{C.Mouhot@dpmms.cam.ac.uk} \end{center}}

\begin{document}
\maketitle                            

\newcommand{\p}{\partial}
\newcommand{\og}{\omega}
\newcommand{\Og}{\Omega}
\newcommand{\Dt}{\Delta}
\newcommand{\ld}{\lambda}
\newcommand{\Ld}{\Lambda}
\newcommand{\Gm}{\Gamma}
\newcommand{\gm}{\gamma}
\newcommand{\vp}{\varphi}
\newcommand{\vep}{\varepsilon}
\newcommand{\ep}{\epsilon}
\newcommand{\vh}{\varrho}
\newcommand{\vap}{\varphi}
\newcommand{\kp}{\eta}
\newcommand{\Sg}{\Sigma}
\newcommand{\fr}{\frac}
\newcommand{\sg}{\sigma}
\newcommand{\ept}{\emptyset}
\newcommand{\btd}{\nabla}
\newcommand{\btu}{\bigtriangleup}
\newcommand{\tg}{\triangle}
\newcommand{\Th}{{\cal T}^h}
\newcommand{\ul}{\underline}
\newcommand{\Ups}{\Upsilon}
\newcommand{\be}{\begin{equation}}
\newcommand{\ee}{\end{equation}}
\newcommand{\ba}{\begin{array}}
\newcommand{\ea}{\end{array}}
\newcommand{\bea}{\begin{eqnarray}}
\newcommand{\eea}{\end{eqnarray}}
\newcommand{\beas}{\begin{eqnarray*}}
\newcommand{\eeas}{\end{eqnarray*}}
\newcommand{\dpm}{\displaystyle }
\newcommand{\intt}{\int\!\!\!\!\int}
\newcommand{\inttt}{\int\!\!\!\!\int\!\!\!\!\int}
\newcommand{\intttt}{\int\!\!\!\!\int\!\!\!\!\int\!\!\!\!\int}
\newcommand{\bR}{{\bf R}^3 }
\newcommand{\bS}{{\bf S}^2 }
\newcommand{\bSS}{{\bS}\times{\bS}}
\newcommand{\bSSS}{{\bS}\times{\bS}\times{\bS}}
\newcommand{\bRR}{{\bR}\times{\bR}}
\newcommand{\bRRR}{{\bRR}\times{\bR}}
\newcommand{\bRS}{{\bR}\times {\bf S}^2 }
\newcommand{\bRRS}{{\bRR}\times{\bf S}^2 }
\newcommand{\bRRRS}{{\bRRR}\times{\bf S}^2 }
\newcommand{\bRP}{{\bf  R}_{+}}
\newcommand{\bRPP}{{\bf  R}_{+}^2}
\newcommand{\bRPPP}{{\bf  R}_{+}^3}
\newcommand{\bRPPPP}{{\bf  R}_{+}^4}
\newcommand{\bT}{{\bf T}^3}
\newcommand{\bTS}{{\bf T}^3\times{\bS}}
\newcommand{\bTR}{{\bT}\times{\bR}}
\newcommand{\bTRR}{{\bT}\times{\bRR}}
\newcommand{\bTRRR}{{\bT}\times{\bRRR}}
\newcommand{\bTRRS}{{\bT}\times{\bRRS}}
\newcommand{\bTRL}{{\bf R}^1\times{\bT}}
\newcommand{\ra}{\rangle}
\newcommand{\rsk}{s, {\rm k}}
\newcommand{\bLRK}{{\bf R}^1\times{\bf K}^3}
\newcommand{\rip}{\rightharpoonup}
\newcommand{\bRN}{{\bf R}^N}
\newcommand{\bRRN}{{\bf R}^N\times{\bf R}^N}
\newcommand{\bRRSN}{{\bf R}^N\times{\bf R}^N\times{\bf S}^{N-1}}
\newcommand{\bRRRN}{{\bf R}^N\times{\bf R}^N\times{\bf R}^N}
\newcommand{\bRRRRN}
{{\bf R}^N\times{\bf R}^N\times{\bf R}^N\times{\bf R}^N}
\newcommand{\bSN}{{\bf S}^{N-1}}
\newcommand{\bRSN}{{\bf R}^N\times{\bf S}^{N-1}}
\newcommand{\bSSN}{{\bf S}^{N-1}\times{\bf S}^{N-1}}

\newtheorem{theorem}{Theorem}
\newtheorem{definition}[theorem]{Definition}
\newtheorem{lemma}[theorem]{Lemma}
\newtheorem{conjecture}[theorem]{Conjecture}
\newtheorem{corollary}[theorem]{Corollary}
\newtheorem{proposition}[theorem]{Proposition}

\def\theThm{{\arabic{section}.\arabic{theorem}}}
\numberwithin{equation}{section}
\numberwithin{theorem}{section}

\theoremstyle{remark}
\newtheorem{remark}[theorem]{Remark}
\newtheorem{remarks}[theorem]{Remarks}
\newtheorem{examples}[theorem]{Examples}
\newtheorem{example}[theorem]{Example}


\begin{abstract}
  We exhibit a large class of Lyapunov functionals for nonlinear
  drift-diffusion equations with non-homogeneous Dirichlet boundary
  conditions. 
  These are generalizations of large deviation functionals
  for  underlying stochastic many-particle systems, the zero range
  process and the Ginzburg-Landau dynamics, which we describe briefly. 
  As an application, we prove linear inequalities between such an entropy-like functional and its
  entropy production functional for the boundary-driven porous medium
  equation in a bounded domain with positive Dirichlet conditions:
  this implies exponential rates of relaxation related to the first
  Dirichlet eigenvalue of the domain. We also 
  derive Lyapunov functions for systems of nonlinear diffusion equations, and for
  nonlinear Markov processes with non-reversible stationary measures.
\end{abstract} 
\bk



\tableofcontents

\section {Introduction}
\label{sec:intro}

\subsection{The setting} 

We consider a diffusion operator on an open subset $\Om$ in a
Riemannian manifold $M$ ($\Om$ may be equal to $M$): Let $A$ be 
a linear map $T_xM\to T_x^*M$, $x \in \Omega$. We 
assume that the associated quadratic form $v\longmapsto \<Av,v\>$ is
nonnegative (on each tangent space), and refer to it as a
\emph{diffusion matrix}. We consider a \emph{field} $E : x \in M \to
T_x M$ and 
define the associated drift-diffusion operator $\cL=\cL_{A,E}$ by
\[ 
\cL_{A,E} \mu := -\nabla^*  (A\nabla\mu + E \mu) \quad \mbox{ for } \quad \mu \
\mbox{ a measure on} \ \Om.
\] 
Here $\nabla^*$ denotes the adjoint operator (i.e. $-\mbox{div}$ in
the case of a flat geometry) and $\nabla\mu$ is defined in the weak
sense; if $\mu = h\vol$ then $\nabla\mu = (\nabla h)\,\vol$ by usual
Riemannian calculus ($\vol$ is the canonical volume measure on $M$).
We shall always assume that $\mu$ has a smooth density, although this
assumption can usually be removed by approximation arguments. 

We then consider a non-negative function $\sigma : \R_+ \to \R_+$ that
is monotonically increasing on $\R_+$, and a measure $\nu$ on
$\Omega$. We define the corresponding \emph{nonlinear drift-diffusion
  equation}
\begin{equation}\label{eq:ibv1}
  \frac{\partial \mu}{\partial t} = \cL_{A,E} \left( \sigma(f) \nu \right), \quad x \in
  \Omega, \quad f = \frac{\dd \mu}{\dd \nu}.
\end{equation}
In the simplest case where $\nu$ is the Lebesgue measure and $A$ the
identity, this is just an equation for the density $f$:
\begin{equation*}
  \frac{\partial f}{\partial t} = \Delta \sigma (f) + \text{div} \big( E \sigma (f) \big). 
\end{equation*}
Note that the diffusion coefficient associated to this equation is
$\sigma^\prime$ and is non-negative as $\sigma$ is assumed to be
monotonically increasing.  In particular for $\sigma(x) = x^m$, we
recover the porous medium equation.

The equation is supplemented with initial conditions and boundary
conditions of \emph{non-homogeneous} Dirichlet type: 
\begin{equation}
  \label{eq:ibv2}
  f_{|t=0} = f_0, \quad f_{|\partial \Om} = f_b
\end{equation}
for some $f_0 \ge 0$ on $x \in \Omega$, and $f_b>0$ on
$x \in \partial \Omega$. This represents an open system in contact at
its boundary with reservoirs having specified density values $f_b(x)$ for $x$ in $\partial \Om$. 

\begin{remark}
  All our results still hold in the simpler case of \emph{generalized}
  Neumann boundary conditions (that takes into account the force field
  $E$)
\begin{equation}\label{ipp}
  \forall g \in C_b(\Omega), \quad \int_\Om g\, \cL_{A,E} \left(
    \sigma(f) \nu \right)  = \int_\Om \<A\nabla\left( \sigma(f)
    \nu\right) + E \sigma(f) \nu ,\nabla g\>,
\end{equation}
that formally is the same as 
\[ \<A\nabla\left( \sigma(f) \nu \right) + E \sigma(f) \nu , n\>=0
\qquad \text{on $\pa\Om$},\] where $n=n(x)$ is the outward normal to
$\Om$. When $E=0$, this reduces to the classical Neumann boundary
conditions. We also recover previous results in the case of
\emph{homogeneous} Dirichlet boundary conditions when $f_b$ is
constant.
\end{remark}

We finally assume that a stationary measure $\dd \mu_\infty = f_\infty
\dd \nu$ exists: 
\begin{equation}
  \label{eq:mu-stat}
  \cL_{A,E} \left( \sigma(f_\infty) \nu\right) = 0.
\end{equation}
Note that in general the measure $\sigma(f_\infty) \nu$ is not
\emph{reversible}: 
\[
A \nabla (\sigma(f_\infty) \nu) + E \sigma(f_\infty) \nu \not \equiv 0,
\]
and $f_\infty$ is not explicit and depends on the boundary conditions
in a non-local manner. This is a manifestation of the open nature of
this system.

\subsection{The main results}
\label{sec:main-results}

Let us introduce, for any $C^2$ function $\Phi : \R_+ \to \R_+$ that
satisfies $\Phi(1)=\Phi'(1)=0$ and $\Phi'' \ge 0$, the following
functional
\begin{equation} \label{eq:Hphi}
  H_\Phi (f|f_\infty) := \int_\Om \left( \int_{f_\infty(x)} ^{f(x)}
    \Phi'\left(\frac{\sigma(s)}{\sigma(f_\infty(x))} \right) \dd s\right) \dd
    \nu(x). 
\end{equation}
Clearly $\Phi(z) \ge 0$ for $z \in \R_+$ with equality at $z=1$, and
$\Phi'(z) \le 0$ for $z \in [0,1]$ and $\Phi'(z) \ge 0$ for $z \in
[1,+\infty)$. Due to the fact that $\sigma$ is monotonically
increasing, the functional $H_\Phi$ is therefore non-negative.

\begin{remark}
Observe that when $\sigma(s) =s$, this functional reduces to 
\begin{equation*}
  H_\Phi(f|f_\infty) = \int_\Om \Phi\left( \frac{f}{f_\infty} \right)
  f_\infty \dd \nu = \int_\Om \Phi\left( \frac{\dd \mu}{\dd
      \mu_\infty} \right) \dd \mu_\infty.
\end{equation*}
When $\Phi(z) = z \ln z-z+1$, one recognises the
\emph{Kullback-Leibler information}
\begin{equation*}
  \int_\Om \dd \left( \mu \ln \frac{\dd \mu}{\dd
      \mu_\infty} - \mu + \mu_\infty \right)
\end{equation*}
which differs from Shannon's relative entropy $\int_\Om \dd ( \mu \ln
\frac{\dd \mu}{\dd \mu_\infty} )$ when $\mu$ and $\mu_\infty$ have
different masses. For lack of a better name, we will use the name
``\emph{relative $\Phi$-entropy}'' for the functional $H_\Phi$. Note
that our definition of entropy has the opposite sign from that used in
statistical mechanics or thermodynamics -- so it is generally
decreasing rather than increasing in time. 
\end{remark}

Let us also introduce, for any $C^2$ function $\Psi : \R \to \R_+$ that
satisfies $\Psi'(0)=0$ and $\Psi'' \ge 0$, the following functional
\begin{equation} \label{eq:Hpsi}
  N_\Psi (f|f_\infty) := \int_\Om \left( \int_{f_\infty(x)} ^{f(x)}
    \Psi'\left(\sigma(s) - \sigma(f_\infty(x)) \right) \dd s\right) \dd
    \nu(x). 
\end{equation}
Clearly $\Psi'(z) \ge 0$ for $z \ge 0$ and $\Psi'(z) \le 0$ for $z \le
0$. Due to the fact that $\sigma$ is monotonically increasing, the
functional $N_\Psi$ is therefore non-negative again.

\begin{remark}
Observe that when $\sigma(s) =s$, this functional reduces to 
\begin{equation*}
  N_\Psi(f|f_\infty) = \int_\Om \Psi\left( f - f_\infty \right)
  \dd \nu = \int_\Om \Psi\left( \frac{\dd \mu}{\dd \nu} - \frac{\dd
      \mu_\infty}{\dd \nu} \right) \dd \nu.
\end{equation*}
When $\Psi(z) = z^p$, $p > 1$, one recognises some kind of Lebesgue
norm
\begin{equation*}
  \int_\Om \left| f - f_\infty \right|^p \dd \nu
\end{equation*}
with respect to the stationary measure. Again for lack of a better
name, we will use the name ``\emph{relative $\Psi$-entrophy}'' for
the functional $N_\Psi$.
\end{remark}

\begin{theorem}
  \label{theo:main1}
  Under the previous assumptions, for any solutions $f_t= \dd \mu_t/
  \dd \nu$ in $L^\infty$ to the nonlinear drift-diffusion equation
  \eqref{eq:ibv1}-\eqref{eq:ibv2}, one has in the sense of
  distribution
  \begin{equation*}
        \dt H_\Phi(f_t|f_\infty) = - \int_\Om \Phi''(h) \left\langle A
      \nabla h, \nabla h \right\rangle \sigma(f_\infty) \dd \nu \le 0, 
    \quad h := \frac{\sigma(f)}{\sigma(f_\infty)} 
  \end{equation*}
  
Moreover if $\nu$ is reversible: $A \nabla \nu + E \nu \equiv 0$, then
\begin{equation*} 
\dt N_\Psi(f_t|f_\infty) = - \int_\Om \Psi''(g) \left\langle A
      \nabla g, \nabla g \right\rangle \dd \nu \le 0, 
    \quad g := \sigma(f) - \sigma(f_\infty).
\end{equation*}
\end{theorem}

\begin{remark} If $\Phi(z) = z \ln z -z +1$ and $\sigma(s) =
  s^m$ then 
\[ 
H_\Phi (f|f_\infty) = m \int_\Omega \left(f
  \ln\frac{f}{f_\infty} - f + f_\infty \right)\,\dd \nu =
m\, \int_\Omega \dd \left( \mu \ln \frac{\dd \mu}{\dd \mu_\infty} - \mu +
  \mu_\infty \right).\] (This is independent of the choice of $\nu$.)

If $\Phi(z) = (z-1)^2/2$ and $\sigma(s) = s^m$, then
$$
H_\Phi (f | f_\infty) = \int_{\Omega} \left( \frac{f^{m+1} - f_\infty
  ^{m+1}}{(m+1) f^{m+1} _\infty} - (f-f_\infty) \right) \, \dd \nu.
$$

If $\Psi(z) = z^2/2$, $\sigma(s) = s^m$ and $\nu$ reversible, then
\begin{equation}\label{eq:enstrophy-pme}
N_\Psi (f | f_\infty) = \int_{\Omega} \left( \frac{f^{m+1} - f_\infty
  ^{m+1}}{(m+1)} - (f-f_\infty) f_\infty^m \right) \, \dd \nu.
\end{equation}

Observe that the $\Phi$-entropies and their entropy production
functionals measure the distance between $f$ and $f_\infty$ through
quotients, whereas the $\Psi$-entrophies and their entropy production
functionals measure this distance through differences. Hence one can
expect that the relation between $\Phi$-entropies and their entropy
production functionals is connected to nonlinear inequalities of
logarithmic Sobolev or Beckner type, whereas the relation between
$\Psi$-entrophies and their entropy production functionals is
connected to \emph{linear} inequalities of Poincar\'e or spectral
theory type.
\end{remark}

\mk

Our second result is an application of these new entropies to the
study of the long-time behavior of the porous medium equation (PME) in
a bounded domain. We assume here that $\Omega$ is bounded, $\sigma(s)
= s^m$, $m \ge 1$, $A$ is the identity and $\nu = \vol$:
\begin{equation}\label{eq:pme}
  \frac{\partial f}{\partial t} = \Delta(f^m) \ \mbox{on} \ \Omega, \quad f_{|t=0} =
  f_0, \quad f_{|\partial \Omega} = f_b, \quad m \ge 1
\end{equation}
with $f_0 \ge 0$ on $\Omega$ and $f_b >0$ and bounded on $\partial \Omega$. 

\begin{theorem}
  \label{theo:main2}
  Consider the relative $\Psi$-entrophy \eqref{eq:enstrophy-pme} with
  $\nu = \vol$,
  then the solution $f_t$, $t \ge 0$ to \eqref{eq:pme} in $L^\infty$
  satisfies 
\begin{equation*}
  \dt N_\Psi(f_t|f_\infty) \le - \lambda N_\Psi(f_t|f_\infty), \quad
  N_\Psi(f_t|f_\infty) \le N_\Psi(f_0|f_\infty) 
  e^{-\lambda t}
\end{equation*}
for some constant $\lambda>0$ depending explicitely on $f_b$ through the bounds on
the invariant measure $f_\infty$ and the first Dirichlet eigenvalue of
the domain $\Omega$. This implies the convergence $f_t \to f_\infty$
in time with exponential rate.
\end{theorem}

\subsection{Heuristic for the entropy structure}
\label{sec:heur-entr-struct}

Our heuristic motivation comes from noting that macroscopic equations
of the form \eqref{eq:ibv1} are meant to describe the coarse-grained
evolution of some underlying microscopic system. The time evolution of
the macroscopic variables evolving according to \eqref{eq:ibv1}
follows those of the microstates corresponding to almost sure values
of the appropriate microscopic functions for the time evolving measure
on the microstates. In particular the stationary values of the
macrovariables have full measure in the stationary state of the
microscopic system. It can then be shown, in cases where
\eqref{eq:ibv1} can be derived from microscopic models, that the large deviation functionals arising from microscopic dynamics are
Lyapunov functions for the macroscopic evolution equation associated
to the microscopic dynamics (cf. \cite{BDGJL,Maes}).  
We refer to Section~\ref{subsec:heuristic} for such a derivation in the case where
the underlying microscopic model evolves according to a simple
stochastic dynamics.

In particular in \cite{BDGJL}, the
relative entropy $H_\Phi$, for $\Phi=z \ln z - z +1$
$$
H_\Phi (f|f_\infty) := \int_\Om \left( \int_{f_\infty(x)} ^{f(x)}
    \ln \left(\frac{\sigma(s)}{\sigma(f_\infty(x))} \right) \dd s\right) \dd x 
$$
was identified with the large deviation functional of the zero range
process, but the focus was not on the PDE itself, but rather on the
stochastic particle system.  In \eqref{eq:Hphi}, we generalize this
functional to derive a new class of Lyapunov functionals for nonlinear PDE.
The functional \eqref{eq:Hpsi} can also be understood as a generalization of the 
large deviation functional associated to the Ginzburg-Landau process.
It is unclear to us whether these functionals are related to the large deviation
functionals arising from some other microscopic dynamics. We also
refer to \cite{MR2842966,Zimmer-bis} where the entropic gradient flow
structure is explored on the basis of large deviation principles for
linear diffusion equations on the real line (without boundary).

Using such simple models is, heuristically at least, justified by the
fact that the macroscopic equations are insensitive to many details of
the microscopic dynamics. This is fortunate since the derivation of
such macroscopic equations from the true microscopic dynamics, be it
Hamiltonian or quantum, is beyond our current mathematical abilities,
cf. \cite{GL,L}. However, even without a rigorous derivation, any
macroscopic equation describing the time evolution of a physical
system must obviously be consistent with the properties of the
microscopic dynamics. Thus, the properly defined macroscopic
(Boltzmann) entropy of an isolated system must be monotone. Isolated
for equation \eqref{eq:ibv1} means Neumann or periodic boundary
conditions, whereas for the nonlinear Boltzmann equation (appropriate
for low density gases) isolated means periodic or specular reflection
at the boundary.  In both examples the entropy Lyapunov functional is
the large deviation functional in the micro-canonical ensemble which
is stationary for the microscopic dynamics of the isolated system.

The existence of Lyapunov functions for the macroscopic equations can
be and has been, studied independently of any underlying microscopic
model.  It seems however desirable conceptually and, as we shall see
later, in some cases also practically to make the connection between
the large deviation functional for the microscopic dynamics and the
macroscopic evolution equations.  Entropy and large deviation
functionals play an important role in the understanding of the
macroscopic evolutions in addition to their intrinsic interest for
microscopic systems.  For references see \cite{GL,L}.

\subsection{Some references on nonlinear diffusions}
\label{sec:some-bibl-mater}

The starting point on the study of entropy structure for diffusion
equations is arguably the seminal papers of Gross \cite{gross} and
Bakry and \'Emery \cite{MR889476}, on the logarithmic Sobolev
inequality. Then the note \cite{MR1447044} exposed the method of
Bakry-\'Emery for proving logarithmic Sobolev inequalities to the
kinetic community. The paper \cite{MR1777035} used such functional
inequalities in order to study the porous medium equation in the whole
space (nonlinearity $\sigma(f) = f^m$ with $m >1$ and $\Omega =
\R^d$). Later the paper \cite{MR1986060} studied the fast diffusion
equation relaxation towards Barenblatt self-similar profiles in the
whole space (nonlinearity $\sigma(f) = f^m$ with $(d-2)/2 < m <
(d-1)/d$). The paper \cite{MR1853037} revisited the whole theory of
logarithmic Sobolev and Poincar\'e inequalities (including the
Holley-Stroock criterion for perturbated potentials), for a general
class of nonlinearity including fast diffusion and porous medium
equations. The paper \cite{MR1842428} treats the case of a convex
domain with Neuman boundary conditions by a penalization method (using
a limiting process with a ``barrier'' confining potential). We also
refer to \cite{MR2065020} for a review of ``entropy - entropy
production methods'' in kinetic theory. Let us also mention the paper
\cite{MR2152502} providing some refinements of Sobolev inequalities,
and the paper \cite{MR2435196} which provides a direct proof in the
case of Neuman boundary conditions (without penalization), but still
with the convexity assumption.
Entropy approaches have also been devised for reaction-diffusion-type systems
\cite{FFM}.

There is an enormous amount of valuable informations in the book of
Vazquez \cite{MR2286292}, centered on the porous medium case
($\sigma(s) = s^m$ with $m>1$). However as all previous works it is
mainly concerned with the whole space problem, and not so much with
entropy structures. The only chapters 19-20 which are concerned with
the asymptotic behavior of the initial-boundary-value problem are
restricted to {\em homogeneous} Dirichlet boundary conditions in a
bounded domain (zero) or in an exterior domain. We also refer to
\cite{brezis} where an entropy was presented for nonlinear diffusions
in bounded domain with homogeneous Dirichlet conditions. Hence it
seems that there are very few works in the PDE community in the case of
a bounded domain with non-homogeneous Dirichlet conditions. This
corresponds to $f_\infty$ describing and ``out-of-equilibrium'' steady
solution.

\subsection{Plan of the paper}

In Section \ref{sec:proof-theor-refth} we prove
Theorem~\ref{theo:main1} and in Section \ref{sec:proof-theor-refth-1}
we prove Theorem~\ref{theo:main2}. Then Section
\ref{sec:syst-nonl-diff} presents an extension of our approach to
systems of two nonlinear diffusion equations and derives  two Lyapunov
functionals in this setting. Section \ref{sec:rks} presents an
extension to linear or pointwise nonlinear Markov processes with
non-reversible stationary measures. Finally in Section
\ref{subsec:heuristic} we explain in details the heuristic that lead
us to introducing these relative entropies and entrophies, on the
basis of the zero range process and the Ginzburg-Landau dynamics.

\section{Proof of Theorem~\ref{theo:main1}}
\label{sec:proof-theor-refth}

Let us first consider the relative $\Phi$-entropy functional
$H_\Phi$. The time derivative is 
\begin{multline*}
  \dt H_\Phi(f|f_\infty) = \int_\Om \Phi'\left(
    \frac{\sigma(f_t(x))}{\sigma(f_\infty(x))} \right) \frac{\dd
    \cL_{A,E}(\sigma(f)
    \nu)}{\dd \nu} \dd \nu \\
  = - \int_\Om \Phi'\left( \frac{\sigma(f_t(x))}{\sigma(f_\infty(x))}
  \right) \nabla^* \left( A \nabla (\sigma(f) \nu) + E \sigma(f) \nu
  \right) \dd \vol.
\end{multline*}

We set $h = \sigma(f) / \sigma(f_\infty)$ and write (using
$\cL_{A,E}(\sigma(f_\infty)\nu)=0$)
\begin{multline*}
   \nabla^* \left( A \nabla (\sigma(f) \nu) + E \sigma(f) \nu
  \right)  =  \nabla^* \left( A \nabla (h \sigma(f_\infty)
  \nu) + E h \sigma(f_\infty) \nu \right) \\ = \nabla^* A \left( (\nabla h)
    \sigma(f_\infty) \nu \right) + \left\langle A
  \nabla(\sigma(f_\infty) \nu ) + E \sigma(f_\infty) \nu, \nabla h \right\rangle.
\end{multline*}

Since $h \equiv 1$ at the boundary $\partial \Omega$ and
$\Phi(1)=\Phi'(1)=0$ we obtain
\begin{multline*}
  \dt H_\Phi(f|f_\infty) = - \int_\Om \Phi'\left( h \right) \nabla^*
  \left( A
    \nabla (h \sigma(f_\infty) \nu) + E h \sigma(f_\infty) \nu \right) \dd \vol \\
  = - \int_\Om \Phi'\left( h \right) \nabla^* A \left( (\nabla h)
    \sigma(f_\infty) \nu \right) \dd \vol \\ - \int_\Om \Phi'\left( h
  \right) \left\langle A
    \nabla(\sigma(f_\infty) \nu) + E \sigma(f_\infty) \nu, \nabla h  \right\rangle \dd \vol \\
  = - \int_\Om \Phi''\left( h \right) \left\langle \nabla h, A \nabla
    h \right\rangle \sigma(f_\infty) \dd \nu \\
  - \int_\Om \left\langle A \nabla(\sigma(f_\infty) \nu) + E
    \sigma(f_\infty) \nu, \nabla \Phi(h) \right\rangle \dd
  \vol \\
  = - \int_\Om \Phi''\left( h \right) \left\langle \nabla h, A \nabla
    h \right\rangle \sigma(f_\infty) \dd \nu \\ - \int_\Om \Phi(h)
  \nabla^* \left( A \nabla(\sigma(f_\infty) \nu) + E \sigma(f_\infty)
    \nu \right)  \dd \vol \\
  = - \int_\Om \Phi''( h ) \left\langle \nabla h, A \nabla h
  \right\rangle \sigma(f_\infty) \dd \nu
\end{multline*}
which concludes the proof (we have used $\cL_{A,E}(\sigma(f_\infty) \nu)=0$ in
the two last lines). 

Let us next consider the relative $\Psi$-entrophy functional
$N_\Psi$. Its time derivative is 
\begin{eqnarray*}
  \dt N_\Psi(f|f_\infty) &=& \int_\Om \Psi'\left( \sigma(f_t(x)) -
    \sigma(f_\infty(x)) \right) \frac{\dd \cL_{A,E}(\sigma(f)
    \nu)}{\dd \nu} \dd \nu \\
  &=& - \int_\Om \Psi'\left( \sigma(f_t(x))- \sigma(f_\infty(x)) \right)
  \\ && \hspace{0.5cm}
  \nabla^* \left( A \nabla (\sigma(f) \nu - \sigma(f_\infty) \nu + E (
    \sigma(f) \nu - \sigma(f_\infty) \nu ) \right) \dd \vol
\end{eqnarray*}
where we have again used  $\cL_{A,E}(\sigma(f_\infty) \nu)=0$ in the last
line. 

We denote $g = \sigma(f) - \sigma(f_\infty)$ and since $g \equiv 0$ at
the boundary $\partial \Omega$ and $\Psi'(0)=0$ we obtain
\begin{multline*}
  \dt N_\Psi(f|f_\infty) = - \int_\Om \Psi'( g ) \nabla^* \left( A
  \nabla (g \nu) + E g \nu \right) \dd \vol \\
  = - \int_\Om \Psi''(g) \left\langle \nabla g, A \nabla g
  \right\rangle \dd \nu - \int_\Om \Psi'( g ) \nabla^* \left( g ( A
  \nabla \nu + E \nu) \right) \dd \vol \\
 = - \int_\Om \Psi''(g) \left\langle \nabla g, A \nabla g
  \right\rangle \dd \nu
\end{multline*}
which concludes the proof (we have used the reversibility of $\nu$ in
the two last lines).

\section{Proof of Theorem~\ref{theo:main2}}
\label{sec:proof-theor-refth-1}



We consider the porous medium equation \eqref{eq:pme} on $\Omega
\subset \R^d$ for $\sigma(s)=s^m$, $m \ge 1$ and $\nu = \vol = \dd x$,
and the relative $\Psi$-entrophy as constructed before with the
choice $\Psi(z) = z^2/2$ (the measure $\nu$ is reversible and
Theorem~\ref{theo:main1} applies). This results in
\begin{multline*}
  N_\Psi(f|f_\infty) = \int_\Omega \left( \int_{f_\infty} ^f \Psi'\left(
      s^m - f_\infty^m\right) \dd s \right) \dd x \\=
  \int_\Omega \left( \int_{f_\infty} ^f \left(
      s^m - f_\infty^m\right) \dd s \right) \dd x \\ 
  = \int_\Omega \left( \frac{f^{m+1} -
        f_\infty^{m+1}}{m+1} - (f-f_\infty) f_\infty^m \right) \dd x.
\end{multline*}

We have from the previous analysis
\begin{equation*}
  \dt N_\Psi(f_t|f_\infty) = - \int_\Omega \left| \nabla (f^m _t - f_\infty ^m)
  \right|^2 \dd x. 
\end{equation*}

We then observe that 
\begin{equation}\label{eq:dirichlet}
  \int_\Omega \left| \nabla g
  \right|^2 \dd x \ge \lambda_D \int_\Omega \left| g
  \right|^2 \dd x
\end{equation}
for any $g \in H^2 _0(\Omega)$ (the Sobolev space with zero boundary
conditions $g_{|\partial \Omega} =0$), where $\lambda_D >0$ is the
first Dirichlet eigenvalue of $\Omega$. We apply this inequality with
$g = f^m -f_\infty^m$.

Finally we perform the following elementary calculation in $\R$: for
any $y_\infty \in K \subset (0,+\infty)$ with $K$ compact, there is a
constant $C_K$ so that
\begin{equation*}
  \forall y \in \R_+, \quad 
  (y^m-y_\infty^m)^2 \ge C_K \left( \frac{y^{m+1} - y_\infty^{m+1}}{m+1} -
    (y-y_\infty) y_\infty^m \right),
\end{equation*}
(recall that $m \ge 1$).  This inequality is proved as follows: Let $x
= y/y_\infty$, then
\begin{equation*}
  \forall x \in \R_+, \quad  \left(1-x^m\right)^2 \ge \frac{C_K y_\infty ^{1-m}}{m+1} \left[
    x^{m+1} - (m+1) x +m \right]
\end{equation*}
by considering the three cases $x \sim +\infty$, $x \sim 1$, and $x
\sim 0$, and then taking the worst constant over $y_\infty \in
K$.

Since $f_\infty$ is valued in a compact subset of $(0,+\infty)$, we
finally deduce
\begin{equation*}
  \dt N_\Psi(f_t|f_\infty) \le - \lambda_D C_K N_\Psi(f_t|f_\infty) =: - \lambda N_\Psi(f_t|f_\infty)
\end{equation*}
which concludes the proof. 

\begin{remarks}
\begin{enumerate}
\item The key ingredient of this proof is
  equation~\eqref{eq:dirichlet}. The existence of the positive
  constant $\lambda_D$ captures the boundary-driven geometry of the
  problem through classical \emph{linear} spectral theories for
  self-adjoint operators with compact resolvent.
\item The underlying gradient-flow structure shows a typical
  \emph{hypocoercive} pattern, combining the sum of a partially coercive
  ``symmetric'' term and a skew-symmetric non-coercive term. The
  understanding of this structure is an interesting open question
  which will be studied in future  works.
\item Another interesting question is the following. In order to
  capture the role played by the boundary in a functional inequality,
  we have used our general framework of $\Psi$-entrophy in order to
  select an entropy producing an entropy production functional lending
  itself to a \emph{linear} study. However one could wonder whether
  this analysis could be performed for the $\Phi$-relative entropy
  with $\Phi(z) = z \ln z - z +1$. The key ingredient to be proved
  would then be a \emph{Dirichlet logarithmic Sobolev inequality}
  \begin{equation*}
    \int_\Omega \frac{|\nabla h|^2}{h} \dd m \ge \lambda_{DLSI}
    \int_\Omega \left( h \ln h - h +1 \right) \dd m
  \end{equation*}
  for $h \ge 0$ with $h_{|\partial \Omega} =1$ and $m$ a probability
  measure on $\Omega$, where $\lambda_{DSLI}>0$ depends on $\mu$. By
  comparison the usual logarithmic Sobolev inequality only yields
\begin{equation*}
  \int_\Omega \frac{|\nabla h|^2}{h} \dd m \ge \lambda_{LSI}
    \int_\Omega h \ln \frac{h}{\int_\Omega h \dd m} \dd m
  \end{equation*}
for some constant $\lambda_{LSI} >0$ depending on $m$. In the latter
inequality the integrand involves the non-local quantity $\int_\Omega
h \dd m$ which is non-constant along the flow of the nonlinear
diffusion equation. These two inequalities coincide when $\int_\Om h
\dd m =1$. 
\end{enumerate}
\end{remarks}

\section{System of nonlinear diffusion equations}
\label{sec:syst-nonl-diff}

In this section, we contruct a relative entropy and a relative
entrophy for a system of two nonlinear equations.

Let $\sigma^1, \sigma^2 : (\bbR^+)^2 \to (\R^+)^2$, so
that the Jacobian matrix of $(s_1,s_2) \mapsto (\sigma^1(s_1,s_2),\sigma^2(s_1,s_2))$ is
definite positive. 

We consider two diffusion operators $L_{A_1}$ and $L_{A_2}$ with two
diffusion matrices $A_1$ and $A_2$ as defined above, and for $\mu(t,x)
= (\mu^1(t,x),\mu^2(t,x))^T$ we define the following nonlinear system
of diffusion equations
\begin{equation}
\label{eq: coupled hydro}
\frac{\partial \mu}{\partial t} = \binom{L_{A_1}(\sigma^1(f)
  \nu)}{L_{A_2}(\sigma^2(f) \nu)}, \quad f=\binom{f^1}{f^2} =\binom{\dd
  \mu^1/ \dd \nu}{\dd \mu^2/ \dd \nu}, \quad x \in \Omega
\end{equation}
with the initial and boundary conditions
\[
f_{|t=0} = f_0 = \binom{f^1_0}{f^2_0},\quad f_{|\partial \Omega} =
\binom{f^1_b}{f^2_b}
\]
for some Borel functions $f_0 ^1, f_0 ^2 \ge 0$ on $\Om$ and $f_b ^1,
f_b ^2 >0$ on $\partial \Om$, and some reference measure $\nu$.  We
assume the existence of a stationary measure (solution to this
elliptic problem) $(\mu_\infty ^{1},\mu^2_\infty) =
(\sigma^{1}(f_\infty) \nu, \sigma(f_\infty) \nu)$, $f_\infty =
(f_\infty ^1, f_\infty ^2)^T$.

\mk

We first consider the case of the relative entropy. We assume that the
nonlinearity functionals satisfy the compatibility relation
\begin{equation}
\label{eq:symmetry1}
\forall s_1,s_2 \in \R_+, \quad 
\partial_2 \ln \sigma^1 (s_1,s_2) = \partial_1 \ln \sigma^2 (s_1,s_2),
\end{equation}
where $\partial_1, \partial_2$ stands for the derivative wrt the first
and second coordinates. These relations correspond, in the two
components zero range process, to a sufficient condition for the
stationary microscopic to be of the product form, see
Section~\ref{subsec:heuristic}. 

We consider $\Phi(z) = z\ln z - z +1$ and 
\begin{multline}
  H_\Phi(f|f_\infty) = \int_\Om \left( \int_{f_\infty ^1}^{f^1}
    \Phi'\left( \frac{\sigma^1(s,f^2)}{\sigma^1(f_\infty)}\right) \dd s  
+ \int_{f_\infty ^2}^{f^2} \Phi' \left( 
\frac{\sigma^2(f_\infty ^1,s)}{\sigma^2(f_\infty)}   \right) \dd s
\right) \dd \nu \\
= \int_\Om \left( \int_{f_\infty ^1}^{f^1}
    \ln \left( \frac{\sigma^1(s,f^2)}{\sigma^1(f_\infty)}\right) \dd s  
+ \int_{f_\infty ^2}^{f^2} \ln \left( 
\frac{\sigma^2(f_\infty ^1,s)}{\sigma^2(f_\infty)}   \right) \dd s
\right) \dd \nu.
\label{eq: 2 species LD}
\end{multline}

Let us denote
\begin{equation*}
  G_{f_\infty} (f^1,f^2) := \left( \int_{f_\infty ^1}^{f^1}
    \ln \left( \frac{\sigma^1(s,f^2)}{\sigma^1(f_\infty)}\right) \dd s  
+ \int_{f_\infty ^2}^{f^2} \ln \left( 
\frac{\sigma^2(f_\infty ^1,s)}{\sigma^2(f_\infty)}   \right) \dd s
\right). 
\end{equation*}
Observe that, thanks to the compatibility relation \eqref{eq:symmetry1} one has 
\begin{equation*}
\left\{
\begin{array}{l}\ds
  \partial_{f^1}  G_{f_\infty} (f^1,f^2) = \ln \frac{\sigma^1(f)}{\sigma^1(f_\infty)} \vs \\  \ds
\partial_{f^2}  G_{f_\infty} (f^1,f^2) = \ln \frac{\sigma^2(f)}{\sigma^2(f_\infty)}.
\end{array}
\right.
\end{equation*}


Hence we obtain, arguing similarly as for the one-component model
\begin{eqnarray*}
  \dt H_\Phi(f_t|f_\infty) &=& - \int_\Omega \frac{\langle A_1 \nabla
    h^1, \nabla h^1 \rangle}{h^1} \sigma^1 (f_\infty(x)) \dd\nu(x) \\
  && \qquad - \int_\Omega \frac{\langle A_2 \nabla
    h^2, \nabla h^2 \rangle}{h^2} \sigma^2 (f_\infty(x))
  \dd\nu(x) \\ 
&=&  - \int_\Omega \langle A_1 \nabla \ln 
    h^1, \nabla \ln h^1 \rangle \sigma^1 (f(x)) \dd\nu(x) \\
  && \qquad - \int_\Omega \langle A_2 \nabla
    \ln h^2, \nabla \ln h^2 \rangle \sigma^2 (f(x))
  \dd\nu(x) \le 0,
\end{eqnarray*}
where we have used the notation
\begin{equation*}
  h^1 := \frac{\sigma^1(f)}{\sigma^1(f_\infty)}, \quad   h^2 := \frac{\sigma^2(f)}{\sigma^2(f_\infty)}.
\end{equation*}

\begin{remark}
  In the case $A_1 = A_2 = \mbox{Identity}$, the diffusion matrix
  associated to the evolution is given by
\begin{equation*}
\bbD(s_1,s_2) = 
 \left(
 \begin{array}{cc}
 \partial_1 \sigma^1 (s_1,s_2) & \partial_2 \sigma^1 (s_1,s_2)\\
 \partial_1 \sigma^2 (s_1,s_2) & \partial_2 \sigma^2 (s_1,s_2)\\
 \end{array}
 \right) \, .
 \end{equation*}
and the conductivity matrix is 
\begin{equation*}
\mathbb S(s_1,s_2) = 
\left(
\begin{array}{cc}
\sigma^1 (s_1,s_2) & 0 \\
0 &  \sigma^2 (s_1,s_2)\\
\end{array}
\right) \, .
\end{equation*}
Then the Hessian matrix of $G_{f_\infty}(s_1,s_2)$ is 
\begin{equation*}
 \mathbb H(s_1,s_2) =  \left(
 \begin{array}{cc}
 \frac{\partial_1 \sigma^1 (s_1,s_2)}{\sigma^1(s_1,s_2)} & \frac{\partial_2
   \sigma^1 (s_1,s_2)}{\sigma^1(s_1,s_2)} \\
  \frac{\partial_1 \sigma^2 (s_1,s_2)}{\sigma^2(s_1,s_2)} & \frac{\partial_2
    \sigma^2 (s_1,s_2)}{\sigma^2(s_1,s_2)} \\
 \end{array}
 \right)
\end{equation*}
and we note that Einstein's relation is satisfied thanks to \eqref{eq:symmetry1}
\begin{equation*}
\bbD(s_1,s_2)  = \mathbb S(s_1,s_2) \; \mathbb H(s_1,s_2)  \, .
\end{equation*}
\end{remark}

\begin{remark}
  It is natural to ask whether this analysis extends to other
  relative $\Phi$-entropies, when $\Phi$ is different from $\Phi(z)=z
  \ln z -z +1$. The compatibility relations \eqref{eq:symmetry1}
  seem however hard to extend since it is only when $\Phi'(z)=\ln z$
  that the ``generalized'' condition
  \begin{equation*}
    \partial_2 \Phi'\left( \frac{ \sigma^1 (s_1,s_2)}{\sigma^1(f_\infty)}
    \right) = \partial_1 \Phi'\left( \frac{\sigma^2 (s_1,s_2)}{\sigma^2(f_\infty)}\right)
  \end{equation*}
becomes independent of the values of $f_\infty$ and therefore
makes sense.
\end{remark}
 \mk

 We next consider the case of the relative entrophy. We assume that
 the nonlinearity functionals satisfy the compatibility relations
\begin{equation}
  \label{eq:symmetry2}
\forall s_1,s_2 \in \R_+, \quad 
\partial_2 \sigma^1 (s_1,s_2) = \partial_1 \sigma^2 (s_1,s_2).
\end{equation}

We consider $\Psi(z) = z^2/2$ and $\nu$ reversible, and
\begin{multline*}
  N_\Psi(f|f_\infty) \\ = \int_\Om \left( \int_{f_\infty ^1}^{f^1}
    \Psi'\left( \sigma^1(s,f^2) - \sigma^1(f_\infty)\right) \dd s  
+ \int_{f_\infty ^2}^{f^2} \Psi' \left( 
\sigma^2(f_\infty ^1,s) - \sigma^2(f_\infty) \right) \dd s
\right) \dd \nu \\
= \int_\Om \left( \int_{f_\infty ^1}^{f^1}
  \left( \sigma^1(s,f^2) - \sigma^1(f_\infty)\right) \dd s  
+ \int_{f_\infty ^2}^{f^2} \left( 
\sigma^2(f_\infty ^1,s) - \sigma^2(f_\infty) \right) \dd s
\right) \dd \nu.
\end{multline*}

Let us denote
\begin{equation*}
  G_{f_\infty} (f^1,f^2) := \left( \int_{f_\infty ^1}^{f^1}
    \left( \sigma^1(s,f^2) - \sigma^1(f_\infty) \right) \dd s  
+ \int_{f_\infty ^2}^{f^2} \left( 
\sigma^2(f_\infty ^1,s) - \sigma^2(f_\infty) \right) \dd s
\right). 
\end{equation*}
Observe that, thanks to the compatibility relations
\eqref{eq:symmetry2}, one has
\begin{equation*}
\left\{
\begin{array}{l}\ds
  \partial_{f^1}  G_{f_\infty} (f^1,f^2) = \sigma^1(f) - \sigma^1(f_\infty) \vs \\  \ds
\partial_{f^2}  G_{f_\infty} (f^1,f^2) = \sigma^2(f) - \sigma^2(f_\infty).
\end{array}
\right.
\end{equation*}

Hence we obtain, arguing similarly as for the one-component model
\begin{equation*}
  \dt N_\Psi(f_t|f_\infty) = - \int_\Omega \langle A_1 \nabla g^1, g^1
  \rangle \dd\nu(x) 
 - \int_\Omega \langle A_2 \nabla g^2, g^2 \rangle \dd\nu(x) \le 0,
\end{equation*}
where we have used the notation
\begin{equation*}
  g^1 := \sigma^1(f) - \sigma^1(f_\infty), \quad g^2 := \sigma^2(f) -
  \sigma^2(f_\infty).
\end{equation*}

\begin{remark}
  It is again natural to ask whether this analysis extends to other
  relative $\Psi$-entrophy, when $\Psi$ is different from
  $\Psi(z)=z^2/2$. The compatibility condition \eqref{eq:symmetry2}
  seems however hard to extend since it is only when $\Psi'(z)=z$
  that the ``generalized'' condition
  \begin{equation*}
    \partial_2 \Psi'\left( \sigma^1 (x,y) - \sigma^1(f_\infty)
    \right) = \partial_1 \Psi'\left( \sigma^2
      (x,y) - \sigma^2(f_\infty) \right)
  \end{equation*}
becomes independent of the values of $f_\infty$ and therefore
makes sense.
\end{remark}

\begin{remark}
  A simple example of nonlinearity functionals satisfying both
  compatibility relations is: $\sigma^1(s_1,s_2) = \sigma^2(s_1,s_2) =
  \varphi(s_1+s_2)$, for some smooth function $\varphi$ on $\R_+$. In
  the particular case $\varphi(z) = z^m$, $m >0$, one can check that
  when $\Psi(z)=z^2/2$ and $\nu =\vol$, the relative entrophy is
  \begin{equation*}
    N_\Psi(f_t|f_\infty) = \int_\Om \left( \frac{\Sigma^{m+1} -
        \Sigma_\infty^{m+1}}{m+1} - (\Sigma-\Sigma_\infty) \Sigma_\infty ^m \right) \dd \vol
  \end{equation*}
  with $\Sigma=f^1_t+f^2_t$ and $\Sigma_\infty = f_\infty ^1 +
  f_\infty ^2$. When $f_\infty$ is positive and bounded and $m \ge 1$,
  it is straightforward to prove a linear inequality between the
  relative entrophy and its entrophy production in the same manner
  as in Theorem~\ref{theo:main2}:
\begin{equation*}
  \dt N_\Psi(f_t|f_\infty) \le - \lambda \, N_\Psi(f_t|f_\infty)
\end{equation*}
where $\lambda$ is related to the first Dirichlet eigenvalue of the
domain $\Omega$.
\end{remark}

\section{Nonlinear Markov processes}
\label{sec:rks}

In this section we consider a measure space $\X$ and a Markov process
defined by a kernel $K=K(y,\dd x)$, which is a measure on $\X$
depending measurably on $y\in\X$. To $K$ is associated an operator
$\cL=\cL_K$ acting on the space of probability measures on $\X$
defined by
\begin{equation}
  \label{LK} \cL_K\mu = 
  \int_{y \in \X}
  K(y,\dd x)\dd \mu(y) - \int_{y \in \X} K(x,\dd y)\dd \mu(x). 
\end{equation}
We assume that there are no problems of integrability, so that all integrals converge,
and Fubini theorem applies whenever needed. In particular,
\begin{equation}\label{Lmu}
  \int_\X \cL_K\mu = \iint_{\X
    \times \X}
  K(y,\dd x) \dd \mu(y) - \iint_{\X \times \X} K(x,\dd y)\dd \mu(x) =0.
\end{equation}

A probability measure $\nu$ is said to be \emph{$K$-stationary}
if $\cL_K\nu =0$.  It is said to be \emph{$K$-reversible} if
\[ 
K(x,\dd y) \dd \nu(x) = K(y,\dd x)\dd \nu(y).
\]
Of course reversibility implies stationarity but the reverse is not
true in general. Scattering operators with non-reversible stationary
measures can be used for modeling open systems. They resemble
nonlinear diffusion with non-homogeneous boundary conditions. In both
cases the invariant measure depends in a non-local manner on the
global geometry of the problem.

An example is the spatially homogeneous linear Boltzmann equation for
the velocity distribution of particles moving among background
particles.  Letting $(x,y) \to (v,v')$ denote the velocities after and
before collisions, the operator models collisions with background
particles at a specified time independent spatially uniform velocity
distribution.  In the simplest case, the collisions are elastic and
the background particles have a Maxwellian distribution with a given
temperature. Then the stationary measure is reversible. Another more
intricate case is when the collisions are inelastic and the background
particles come off with a non-Maxwellian distribution. The stationary
measure is then non-reversible.

\begin{remark} If $K(y,\dd x)$ is absolutely continuous with respect
  to $\dd \nu(x)$, with density $K(y,x)$, then reversibility for $\nu$
  means
  \[ K(x,y)=K(y,x),\quad \text{$\nu \otimes \nu$-almost surely}. \] In
  physical language, this means that the dynamics satisfy detailed
  balance with respect to the stationary measure $\nu$.
\end{remark}

We consider now a nonlinear function $\sigma : \R_+ \to \R_+$ that
is strictly monotonically increasing, and a finite measure $\nu$ on $\X$. We
define the corresponding \emph{nonlinear scattering equation}
describing a nonlinear Markov process:
\begin{equation}
  \label{eq:NLM}
  \frac{\partial \mu}{\partial t} = \cL_K \left( \sigma(f) \nu \right), \quad f :=
  \frac{\dd \mu}{\dd \nu}
\end{equation}
complemented with initial conditions $f_{|t=0} = f_0$. 

We introduce, similarly as before, for any $C^2$ function $\Phi : \R_+
\to \R_+$ that satisfies $\Phi(1)=\Phi'(1)=0$ and $\Phi'' \ge 0$, the
following functional
\begin{equation} \label{eq:HphiX}
  H_\Phi (f|f_\infty) := \int_\X \left( \int_{f_\infty(x)} ^{f(x)}
    \Phi'\left(\frac{\sigma(s)}{\sigma(f_\infty(x))} \right) \dd s\right) \dd
    \nu(x). 
\end{equation}

Let us prove
\begin{proposition}
  \label{prop:scattering}
  For any solution $\mu_t = f_t \nu \ge 0$, $f_t \in L^\infty$ to the
  nonlinear scattering equation \eqref{eq:NLM}, one has in the sense
  of distributions
\begin{multline*}
        \dt H_\Phi(f_t|f_\infty) \\ = 
         - \iint_{\X \times \X} \left[ (h(y) - h(x))
          \Phi'(h(y)) + \Phi(h(x)) - \Phi(h(y)) \right] \\ 
K(x,\dd y) 
        \sigma(f_\infty) \dd \nu(x) \le 0,  
\end{multline*}
with the notation
\begin{equation*}
  h := \frac{\sigma(f)}{\sigma(f_\infty)}. 
\end{equation*}

Moreover when furthermore $\nu$ is $K$-reversible: 
\[
K(x,\dd y) \dd \nu(x) = K(y,\dd x) \dd \nu( y),
\]
then $N_\Psi$, defined in \eqref{eq:Hphi}, satisfies
\begin{multline*}
  \dt N_\Psi(f_t|f_\infty) \\ = - \int_\X \left[ (g(y) - g(x))
    \Psi'(g(y)) + \Psi(g(x)) - \Psi(g(y)) \right] K(x,\dd y) \dd \nu(x) \le 0
\end{multline*}
with the notation
\begin{equation*}
  g := \sigma(f) - \sigma(f_\infty).
\end{equation*}

Note that the integrands of the right hand sides are non-positive due
to the convexity of $\Phi$ and $\Psi$. 
\end{proposition}

\begin{proof}[Proof of Proposition \ref{prop:scattering}]
We first consider the case of the relative $\Phi$-entropy $H_\Phi$. 
We calculate the time derivative 
\begin{multline*}
  \dt H_\Phi(f_t|f_\infty) = \int_\X \Phi'(h(x)) \cL_K( \sigma(f) \nu) = 
\int_\X \Phi'(h(x)) \cL_K (h \sigma(f_\infty) \nu) \\ 
 = \iint_{\X
    \times \X} \Phi'(h(x)) K(y,\dd x) h(y) \sigma(f_\infty(y)) \dd
  \nu(y) \\ - \iint_{\X \times \X} \Phi'(h(x)) K(x,\dd
  y) h(x) \sigma(f_\infty(x)) \dd \nu(x) \\
 = \iint_{\X \times \X} h(x) \left( \Phi'(h(y)) - \Phi'(h(x)) \right)
 K(x,\dd y) \sigma(f_\infty(x)) \dd \nu(x).
\end{multline*}
Then we use the $K$-stationarity of $\sigma(f_\infty) \nu$ to deduce
\begin{multline*}
\iint_{\X \times \X} h(x) \Phi'(h(x)) K(x,\dd y)
\sigma(f_\infty(x)) \dd \nu(x) \\= 
\iint_{\X \times \X} h(y) \Phi'(h(y)) K(x,\dd y) \sigma(f_\infty(x))
\dd \nu(x)
\end{multline*}
and 
\begin{multline*}
\iint_{\X \times \X} \Phi(h(x)) K(x,\dd y)\sigma(f_\infty(x)) \dd
\nu(x) \\ = 
\iint_{\X \times \X} \Phi(h(y)) K(x,\dd y) \sigma(f_\infty(x)) \dd \nu(x).
\end{multline*}
This allows to rewrite the time derivative as 
\begin{multline*}
\dt H_\Phi(f_t|f_\infty) = - \iint_{\X \times \X} \Bigg[ \left( h(y) - h(x)
  \right) \, \Phi'(h(y)) \\+ \Phi(h(x)) - \Phi(h(y)) \Bigg] K(x,\dd
y) \sigma(f_\infty(x)) \dd \nu(x) \le 0
\end{multline*}
and concludes the proof. 

 We next consider the relative $\Psi$-entrophy $N_\Psi$. Arguing
 similarly we calculate
 \begin{multline*}
   \dt N_\Psi(f_t|f_\infty) = \int_\X \Psi'(g(x)) \cL_K( \sigma(f)
   \nu) =
   \int_\X \Psi'(g(x)) \cL_K (g \nu) \\
   = \iint_{\X \times \X} \Psi'(g(x)) K(y,\dd x) g(y) \dd \nu(y) \\ -
   \iint_{\X \times \X} \Psi'(g(x)) K(x,\dd
   y) g(x) \dd \nu(x) \\
   = \iint_{\X \times \X} g(x) \left( \Psi'(g(y)) - \Psi'(g(x))
   \right) K(x,\dd y) \dd \nu(x).
\end{multline*}
Then we use the $K$-reversibility of $\nu$ to deduce
\begin{equation*}
\iint_{\X \times \X} g(x) \Psi'(g(x)) K(x,\dd y) \dd \nu(x) = 
\iint_{\X \times \X} g(y) \Psi'(g(y)) K(x,\dd y) \dd \nu(x)
\end{equation*}
and
\begin{equation*}
  \iint_{\X \times \X} \Psi(g(x)) K(x,\dd y) \dd \nu(x) = \iint_{\X
    \times \X} \Psi(g(y)) K(x,\dd y) \dd \nu(x).
\end{equation*}
This allows to rewrite the time derivative as 
\begin{multline*}
\dt N_\Psi(f_t|f_\infty) = - \iint_{\X \times \X} \Bigg[ \left( g(y) - g(x)
  \right) \, \Psi'(g(y)) \\+ \Psi(g(x)) - \Psi(g(y)) \Bigg] K(x,\dd
y) \dd \nu(x) \le 0
\end{multline*}
and concludes the proof. 
\end{proof}

\begin{remark}
  In the case of the $\Phi$-entropy $H_\Phi$ with $\Phi(z)=z \ln z -z
  +1$ an alternative argument is the following: using the exchange of
  $x$ and $y$ we have
\begin{align*}
  \dt H_\Phi(f_t|f_\infty) & = \iint_{\X
    \times \X} \ln h(x) K(y,\dd x) h(y) \sigma(f_\infty(y)) \dd \nu(y) \\
  & \qquad \qquad - \iint_{\X \times \X} \ln h(x) K(x,\dd y) h(x)
  \sigma(f_\infty(x)) \dd
  \nu(x) \\
  & = - \iint_{\X \times \X} h(x) \ln \frac{h(x)}{h(y)}\,K(x,\dd
  y) \sigma(f_\infty(x)) \dd \nu(x)\\
  & = - H(F\pi | G \pi) \leq 0
\end{align*}
Here $H(F\pi | G \pi)$ is the relative entropy between $F\pi$ and $G\pi$ with
\[
\dd \pi(x,y) = K(x,\dd y)\sigma(f_\infty(x)) \dd \nu(x),  \quad 
F(x,y)=h(x), \quad G(x,y)=h(y).
\]
By $K$-stationarity of $\sigma(f_\infty) \nu$,
\begin{multline*} 
  \iint_{\X \times \X} F \dd \pi(x,y)  = \iint_{\X \times \X} h(x)
  K(x,\dd y) \sigma(f_\infty(x)) \dd \nu(x) \\
  = \iint_{\X \times \X} h(y) K(x,\dd y)\sigma(f_\infty(x)) \dd \nu(x) = \iint_{\X
    \times \X} G \dd \pi(x,y)
\end{multline*}
and since the measures $F\pi$ and $G \pi$ have same mass, their
relative entropy is non-negative. 
\end{remark}

\begin{remark} When $\sigma(f_\infty) \nu$ is $K$-reversible, there is
  a simpler proof in the case of the relative $\Phi$-entropy: since
  $K(x,\dd y)\sigma(f_\infty(x)) \dd \nu(x)$ is invariant under the
  exchange of $x$ and $y$, we have
\begin{multline*} 
\dt H_\Phi(f_t|f_\infty) = \\ -
\frac12 \iint_{\X \times \X} \bigl[h(x)-h(y)\bigr]\,\left[ \Phi'(h(x)) -
  \Phi'(h(y))\right] K(x,\dd y)\sigma(f_\infty(x)) \dd \nu(x)
\end{multline*}
which is obviously non-positive since $(a-b) (\Phi'(a) - \Phi'(b))\geq
0$ for all $a,b\geq 0$ due to the convexity of $\Phi$.
\end{remark}


\section{Microsopic heuristic derivation of relative entropies}
\label{subsec:heuristic}

In this section, we recall some facts on the zero range process
(ZRP) \cite{Spitzer, MR2145800} to provide a heuristic explanation for the
Lyapunov function \eqref{eq:Hphi}.  
The relation between large deviations and Lyapunov functionals has  been 
investigated in \cite{BDGJL, Maes, KL}.
In particular we stress the fact that the
Lyapunov function \eqref{eq:Hphi} was already computed in \cite{BDGJL}.
We will also recall the definitions of the multi-type zero range process 
which is related to the functional \eqref{eq: 2 species LD}
and of the Ginzburg-Landau dynamics which is associated to the functional
\eqref{eq:Hpsi}.

\subsection{Zero range process}

The ZRP is a lattice gas model with a conservative stochastic dynamics
\cite{KL} which we define below on the microsocopic domain $\Omega_N = \{
1,N \}^d$.  The microscopic jump rates of the dynamics are determined
by a function $\mathsf g : \bbN \to \bbR^+$ such that $\mathsf g(0) = 0$ and, for the
sake of simplicity, we assume that $\mathsf g$ is increasing. In the
following, $\mathsf g$ will be related to the diffusion coefficient $\gs$ of
the diffusion equation \eqref{eq:ibv1}.

To each site $i$ of $\Omega_N$, one associates an integer variable
$\eta_i$ which specifies the number of particles at this site. The
configuration $\eta(t) = \{ \eta_i (t) \}_{i \in \Omega_N}$ evolves
according to a stochastic dynamics.  Given $\eta(t)$ at time $t$, a
random variable $\tau_i$ with exponential rate $N^2 \mathsf g (\eta_i (t) )$ is
associated at each site.  Let $\tau = \min_i \tau_i$.  The
configuration remains unchanged until time $t + \tau$, then the site
which has the smallest $\tau_i$ is updated: a particle at site $i$
jumps randomly to one its neighbors, say $j$ and the new configuration
becomes
$$
\eta_i (t+\tau) = \eta_i (t) -1; \quad 
\eta_j (t+\tau) = \eta_j (t) +1; \quad
\eta_k (t+\tau) = \eta_k (t), \quad   k \not =  i,j\, . 
$$
Note that if $\eta_i = 0$ then $\tau_i = \infty$ because $ \mathsf g (0) =0$ and
therefore the site cannot be selected. After this update, new
variables $\tau_i$ are drawn with rates $N^2 \mathsf g (\eta_i (t+\tau) )$ and
the same rules apply to the next updates.  The case of independent
random walks is given by $\mathsf g(n) = n$: each walk evolves at
rate $1$ and therefore a site with $\eta_i$ particles will be updated
at rate $\eta_i$.

\subsection{The invariant measure}
\label{sec: The invariant measure}

The previous dynamics preserve the number of particles and describe an
isolated system with an explicit invariant measure.  Let $\mathsf m^\gl$ be
the probability measure on the integers given by
\begin{eqnarray}
\label{eq: mes micro ZR}
  \forall k \in \bbN^*, \qquad \mathsf m^\gl ( k) 
  = \frac{1}{Z_\gl} \;  \frac{\gl^k }{\mathsf g (1) \cdots \mathsf g (k)},
\end{eqnarray}
with $\mathsf m^\gl (0) = \frac{1}{Z_\gl} $ and the normalization
constant
\begin{eqnarray*}
Z_\gl = 1 + \sum_{k = 1}^\infty  \frac{\gl^k}{\mathsf g(1) \cdots \mathsf g(k)}.
\end{eqnarray*}
Define $\gp(\gl) = \ln Z_\gl$, then the mean density of particles in the stationary state is
obtained as
\begin{eqnarray}
\label{eq: lambda micro}
\bbE_{\mathsf m^\gl} \big( \eta \big) = \sum_{k=1} ^\infty \;  k \mathsf m^\gl ( k) = \gl \gp^\prime (\gl)  \, .
\end{eqnarray}
Thus any mean density $f$ can be recovered by tuning $\gl$
appropriately $\gl = \gl(f)$ such that $f = \bbE_{\mathsf m^\gl}
\big( \eta \big)$.  In the following, we will index the measure by its
mean density $f$ and use the notation $\mathsf m_f = \mathsf
m^{\gl(f)}$.  We shall see later that the conductivity $\gs (f)$  at density
$f$ can be interpreted as the expectation of $\mathsf g$
\begin{equation}
\label{eq: sigma micro}
\gs ( f) = \bbE_{\mathsf m_f} \big( \mathsf g( \eta) \big) = \gl(f).
\end{equation}
The product measure 
$$
\mathsf m_{f,N} (\eta) = \bigotimes_{i \in \Omega_N} \mathsf m_f (\eta_i)
$$
is invariant for the ZRP and the measure $\mathsf m_{f,N}$ conditioned to
a fixed number of particles is also invariant (see \cite{Spitzer}).

\subsection{Static large deviations}
\label{sec: Static large deviations}

Given the mean density $f_\infty \in (0,1)$, we compute now the large
deviations of the measure $\mathsf m_{f_\infty,N}$.  The parameter $\gl$
is determined by \eqref{eq: lambda micro} and the relation \eqref{eq:
  sigma micro} implies that $\gs(f_\infty) = \gl$.  We will first check
that
\begin{lemma}
\label{lem: total mean}
The large deviation function of the total density 
\[
S_N = \frac{1}{N^d} \sum_{i \in \Omega_N} \eta_i 
\]
is given by
\begin{equation}
\label{eq: LD mean}
\forall f \in [0,1], \qquad 
\lim_{\gep \to 0} \lim_{N \to \infty}
\; \frac{1}{N^d} \ln \bbP_{\mathsf m_{f_\infty,N}} \Big( S_N \in [f -\gep, f+\gep] \Big)
=  F(f | f_\infty),
\end{equation}
where 
\begin{equation}\label{Fra}
F(f | f_\infty) := \int_{f_\infty} ^f \ln \left( \frac{\sigma(s)}{\sigma(f_\infty)}
\right) 
  \dd s.
\end{equation}
\end{lemma}

\begin{proof}
  The large deviation function can be obtained as the Legendre
  transform of the exponential moments \cite{DZ}
\begin{equation*}
  \psi(\gga) = \lim_{N \to \infty}
\; \frac{1}{N^d} \ln \bbE_{\mathsf m_{f_\infty,N}} \Big( \exp( \gga N^d \, S_N) \Big).
\end{equation*}
As $\mathsf m_{f_\infty,N}$ is a product measure, the previous expression
factorizes and reduces to the computation on one site
\begin{equation}
\label{eq: laplace transform}
\psi(\gga) = \ln \bbE_{\mathsf m_{f_\infty}} 
\Big( \exp( \gga \eta) \Big) = \gp(  \exp( \gga) \gl) - \gp(\gl).
\end{equation}
The large deviation function $\cF(f)$ is the Legendre transform of $\psi$
\begin{equation}
\cF(f) = \sup_\gga \Big\{ \gga f - [ \gp(  \exp( \gga) \gl) - \gp(\gl) ]  \Big\}.
\end{equation}
The supremum is reached for $\gga^*$ such that 
\begin{equation*}
  \cF^\prime (f) = \gga^*, \qquad  f = \exp( \gga^*) \gl \; \gp^\prime (  \exp( \gga^*) \gl).
\end{equation*}
The last equality combined with \eqref{eq: lambda micro} and
\eqref{eq: sigma micro} implies that $\gs(f) = \exp( \gga^*) \gl$.  As
$\gl = \gs(f_\infty)$, we deduce that 
\[
\cF (f) = \int_{f_\infty} ^f \ln \left( \frac{\gs(s) }{\gs(f_\infty)} \right) \dd s = F(f |
f_\infty).
\]
This completes \eqref{eq: LD mean}.
\end{proof}

Let us discretize the unit cube $\Omega = [0,1]^d$ in $\mathbb{R}^d$ with a mesh $1/N$ and
embed $\Omega_N$ in $\Omega$.  The local density of the microscopic system
can be viewed as an approximation of a density function $f(x)$ in
$\Omega$.  To make this quantitative, let us introduce $\pi$ the
empirical measure associated to the microscopic configuration $\eta =
\{ \eta_i\}_i$
\begin{equation}
\label{eq: empirical}
\pi = \frac{1}{N^d} \sum_{i \in \Omega_N} \eta_i \; \delta_{\frac{i}{N}} \, ,
\end{equation}
and we say that $\eta$ approximates the density profile $f(x)$ if
$\pi$ is close to $f(x) \dd x$ in the weak topology.  By abuse of
notation, $\dist(\eta, f)$ stands for the distance (for the weak
topology) between $\pi$ and $f(x) \dd x$.

The large deviations of the local density are given by
\begin{lemma}
\label{lem: LD open}
Let $f$ be a smooth density profile in $\Omega$ then
\begin{equation}
\label{eq: LDP}
\lim_{\gep \to 0} \lim_{N \to \infty} \; -  \frac{1}{N^d} 
\ln  \mathsf m_{f_\infty,N} \left( \big\{  \dist(\eta,f) \leq \gep  \big
  \} \right)  
= \int_{\Omega}  F\bigl(f(x)| f_\infty \bigr) \dd x.
\end{equation}
\end{lemma}
As $\mathsf m_{f_\infty,N}$ is a product measure, the local deviations of the
density in a sub-domain can be determined by the same computations as
in Lemma \ref{lem: total mean}.  The product structure of the measure
implies that the costs of the local deviations add up and therefore
the large deviation function is the integral of the local costs over
the domain.

Observe that the functional in \eqref{eq: LDP} is a particular case of
the Lyapunov functionals $H_\Phi$ introduced in \eqref{eq:Hphi}, when
$\Phi(z) = z \ln z -z +1$, and for Neumann or homogeneous Dirichlet
boundary conditions.

\subsection{Hydrodynamic limit}
\label{sec: Hydrodynamic limit}

We sketch below a heuristic derivation of the hydrodynamic limit for
the ZRP and refer the reader to \cite{KL} for the proofs.

The initial data is chosen such that the microscopic configuration
approximates a smooth macroscopic density profile $f_0(x)$ with $x
\in \Omega$.  For example, the initial data can be sampled from the
product measure
\begin{eqnarray}
\label{eq: initial data}
\mathsf m_{f_0,N} (\eta) = \bigotimes_{i \in \Omega_N} 
\mathsf m_{f_0(\frac{i}{N})}  (\eta_i).
\end{eqnarray}

As $f_0$ is not constant, this measure is not invariant for the
dynamics and the evolution of the local density can be recorded by
$\bbE_{\mathsf m_{f_0,N}} ( \eta_i (t))$. The microscopic evolution rules
lead to
\begin{equation}
\label{eq: hydro micro}
\partial_t \bbE_{\mathsf m_{f_0,N}} ( \eta_i (t)) =
\frac{N^2 }{2d} \sum_{j \sim i} \; \bbE_{\mathsf m_{f_0,N}} \big( \mathsf g( \eta_j (t)) \big) -  
\bbE_{\mathsf m_{f_0,N}} \big( \mathsf g( \eta_i (t)) \big),
\end{equation}
where the sum is over the sites $j$ which are neighbors to $i$. If
$\mathsf g$ is not linear the above equations are not closed and
cannot be solved exactly. However, the dynamics equilibrate very fast
locally and the density is the only slow mode. Thus one expects that
for $i \approx N x$, the local density $f(x,t)$ is the only relevant
parameter and
\begin{equation*}
  \partial_t f(x,t) = \partial_t \bbE_{\mathsf m_{f_0,N}} ( \eta_i (t)), \qquad
  \gs \big( f (x,t) \big) = \bbE_{\mathsf m_{f_0,N}} \big( \mathsf g( \eta_i (t)) \big),
\end{equation*}
where the conductivity $\gs$ is introduced in \eqref{eq: sigma micro}. The
microscopic equations \eqref{eq: hydro micro} can be understood as a
discrete Laplacian which approximates the macroscopic equation
\begin{equation}
\label{eq: hydro macro}
\partial_t f(x,t) = \Delta \gs \big( f (x,t) \big).
\end{equation}
Note that the microscopic density, given by $f( i/N , t)$,
is slowly varying so that the discrete Laplacian in \eqref{eq: hydro
  micro} is of order $1/N^2$ which is exactly compensated by the extra
factor $N^2$ of the microscopic jump rates.  We remark that the
rigorous derivation \cite {KL} of \eqref{eq: hydro macro} requires
further assumptions on the function $\mathsf g$.

In the derivation of \eqref{eq: hydro macro}, we have omitted the
boundary conditions for simplicity.  A more careful computation would
lead to Neumann conditions in order to take care of the fact that the
particles of the ZRP cannot exit the domain $\Omega_N$.

\subsection{Boundary conditions}

The hydrodynamic limit \eqref{eq: hydro macro} with Dirichlet boundary
conditions can be obtained by taking into account the contribution of
reservoirs of particles placed at the boundary of the domain.  This is
achieved at the microscopic level by introducing new exponential
random variables $\{T_i \}$ with rates $N^2 \gga_i$ for any site $i$
at the boundary of the domain $\Omega_N$.  The microscopic dynamics
evolve as described previously with updates according to the times
$\{\tau_i \} \cup \{T_i \}$. If an update occurs at a time $\tau_i$
for a site $i$ at the boundary, a particle jumps uniformly over the
nearest neighbors of $i$, if the jump occurs outside $\Omega_N$ the
particle is removed.  If an update occurs at a time $T_i$, a particle
is added at site $i$.  This mimics the role of reservoirs acting at
the boundary and maintaining locally a constant density of particles.
For any regular density $f_b(x)$ on the boundary $\partial \Omega$ of the
macroscopic domain $\Omega$, the parameters $\gga_i$ can be tuned such
that the hydrodynamic limit satisfies at any time $t>0$
\begin{equation}
\label{eq: hydro macro dirichlet}
 \forall \, x \in  \Omega, \qquad
\partial_t f(x,t) = \Delta \gs \big( f (x,t) \big); \qquad f_{|\partial \Om} = f_b
\end{equation}
which is the same as equations \eqref{eq:ibv1}-\eqref{eq:ibv2} when $A
= \mbox{Identity}$. 

For general boundary conditions $f_b$, a flux of particles is induced
by the reservoirs and the stationary state $f_\infty(x)$ is no longer
a constant but satisfies
\begin{equation}
\label{eq: stationary dirichlet}
 \forall \, x \in  \Omega, \qquad
\Delta \gs \big( f_\infty (x) \big) = 0; 
\qquad f_{|\partial \Omega} = f_b.
\end{equation}

For general microscopic dynamics maintained out of equilibrium by
reservoirs, the stationary measure is unknown as the reversibility of
the dynamics is broken. However for the ZRP, it has been found in
\cite{Demasi,BDGJL} that the invariant measure is a product measure
with a varying density
\begin{eqnarray}
\label{eq: invariant open}
\mathsf m_{f_\infty,N} (\eta) = \bigotimes_{i \in \Omega_N} \mathsf
m_{f_\infty( \frac{i}{N} )} (\eta_i),
\end{eqnarray}
where $f_\infty$ solves \eqref{eq: stationary dirichlet}. Thus the large
deviations for this new stationary measure can be obtained as in Lemma
\ref{lem: LD open}
\begin{lemma}
\label{lem: LD open 2}
Let $f$ be a smooth density profile in $\Omega$ then 
\begin{equation}
\label{eq: LDP open}
\lim_{\gep \to 0} \lim_{N \to \infty} \; -  \frac{1}{N^d} 
\ln  \mathsf m_{f_\infty,N} \left( \big\{  \dist(\eta,f) \leq \gep
  \big \} \right)  = 
\int_{ \Omega}  F\bigl(f(x)|f_\infty(x) \bigr) \dd x.
\end{equation}
\end{lemma}
The functional in \eqref{eq: LDP open} coincides with the Lyapunov
functional $H_\Phi$ introduced in \eqref{eq:Hphi} with $\Phi(z) = z
\ln z -z +1$ for general Dirichlet boundary conditions, when $A =
\mbox{Identity}$ and the domain is an open set of $\R^d$.

\subsection{Asymmetric evolution}

Hydrodynamic equations with an asymmetry $E(x) = (E^{(1)}(x), \dots
, E^{(d)}(x))$, $x \in \Omega$, can be derived by modifying the
microscopic dynamics as follows.  The field $E(x)$ is chosen to be
smooth.  To each site $i$ of $\Omega_N$, one associates the slow
varying field 
\[
E_i = (E^{(1)}(i/N), \dots , E^{(d)}(i/N)).
\]
We introduce the normalization constants
$$
\forall \, i \in \Omega_N, \quad
Z_{i,N} = \sum_{\ell =1}^d  \exp \left( \frac{1}{2 N} E^{(\ell)}(i/N) \right)
+ \exp \left( - \frac{1}{2 N} E^{(\ell)}(i/N) \right).
$$
A particle at site $i$ will jump after a random exponential time with rate 
$N^2 \mathsf g (\eta_i (t)) Z_{i,N}$, but the jump is no longer uniform on the neighboring sites.  
A jump in the direction $\pm \vec{e}_\ell$ occurs with probability
$$
\frac{1}{Z_{i,N}} \; \exp \left( \pm \frac{1}{2 N} E^{(\ell)}(i/N)
\right).
$$
Note that the rates are weakly biaised by a factor of order $1/N$.

As a consequence the microscopic evolution equation \eqref{eq: hydro
  micro} becomes
\begin{multline*}
  \partial_t \bbE_{\mathsf m_{f_0,N}} ( \eta_i (t)) \\
  = N^2 \sum_{\ell = 1}^d \sum_{s = \pm 1} \exp \left( - \frac{s}{2 N}
    E^{(\ell)} \left( \frac{i + s \vec{e}_\ell}{N} \right) \right)
  \bbE_{\mathsf m_{f_0,N}} \left( \mathsf g( \eta_{i + s \vec{e}_\ell} (t)) \right) \\
  - \exp \left( \frac{s}{2 N} E^{(\ell)} \left(\frac{i}{N} \right)
  \right)
  \bbE_{\mathsf m_{f_0,N}} \left( \mathsf g( \eta_i (t)) \right)\\
  \simeq N^2 \sum_{\ell = 1}^d \sum_{s = \pm 1}
  \bbE_{\mathsf m_{f_0,N}} \left( \mathsf g( \eta_{i + s \vec{e}_\ell}
  (t)) \right)
  - \bbE_{\mathsf m_{f_0,N}} \left( \mathsf g( \eta_i (t)) \right) \\
   + \frac{N}{2} \sum_{\ell = 1}^d \left[ E^{(\ell)}
    \left(\frac{i - \vec{e}_\ell}{N} \right) \bbE_{\mathsf m_{f_0,N}}
    \left( \mathsf g( \eta_{i - \vec{e}_\ell} (t)) \right) - E^{(\ell)}
    \left(\frac{i+ \vec{e}_\ell}{N} \right) \bbE_{\mathsf m_{f_0,N}}
    \left( \mathsf g( \eta_{i + \vec{e}_\ell} (t)) \right) \right],
\end{multline*}
where the second equality is obtained by expanding the weak asymmetry
to first order in $1/N$.  In the discrete evolution equation, one can
identify discrete derivatives (which have a contribution of order
$1/N$) and a discrete Laplacian (which has a contribution of order
$1/N^2$).  The weak asymmetry has been tuned so that in the limit $N
\to \infty$, one gets as in \eqref{eq: hydro macro}
\begin{equation*}
\partial_t f(x,t) = \Delta \gs \big( f (x,t) \big) - \div \big( E (x) \gs \big( f (x,t) \big) \big).
\end{equation*}
We refer to \cite {KL} for a rigorous derivation.

\subsection{Multi-species ZRP}

In \cite{MR2145800} a generalization of the ZRP to two-species has
been proposed.  At each site $i$, we denote by $n_i$ the number of
particles of species $A$ and $m_i$ the number of particles of species
$B$.  The jump rates at site $i$ are now given by $\mathsf u
(n_i,m_i)$ for the species $A$ and $\mathsf v (n_i,m_i)$ for the
species $B$.  The conditions imposed on the rates for the stationary
measure to be factorised are
\begin{equation}
\label{eq: micro condition}
\frac{\mathsf u (n_i,m_i)}{\mathsf u (n_i,m_i-1)} = 
\frac{\mathsf v (n_i,m_i)}{\mathsf v (n_i - 1,m_i)} \, .
\end{equation}
At each site, the counterpart of the stationary measure \eqref{eq: mes
  micro ZR}, is now given by
\begin{eqnarray*}
  \forall \, k, \ell \in \bbN^*, \qquad \mathsf m^{\gl, \gga} (k,\ell) 
  = \frac{1}{Z_{\gl, \gga}} \;  \frac{\gl^k }{\mathsf u (1, \ell) \cdots \mathsf u (k, \ell)} \;
 \frac{\gga^k }{\mathsf v (0, 1) \cdots \mathsf v (0, \ell)},
\end{eqnarray*}
where $Z_{\gl, \gga}$ is a normalization factor.
The case of $k$-species ($k > 2$) can also be defined under similar assumptions (see \cite{MR2045171}).

Following the same heuristics as in section \ref{sec: Hydrodynamic
  limit}, one recovers the hydrodynamic equations \eqref{eq: coupled
  hydro} when the diffusion matrices are the identity:
\begin{equation}
\frac{\partial}{\partial t} \binom{f^1}{f^2} = \Delta \, \binom{\gs^1 (f^1,f^2)}{\gs^2 (f^1,f^2)}
\end{equation}
At the macroscopic level, the constraint \eqref{eq: micro condition}
leads to the condition \eqref{eq:symmetry1} on $\sigma_1, \sigma_2$
and the large deviation functional is given by \eqref{eq: 2 species
  LD} with $\Phi(z) = z\ln z - z +1$. A rigorous derivation of the
hydrodynamic equations (in the asymmetric regime) has been achieved in
\cite{MR2045171}.

\subsection{Ginzburg-Landau dynamics}

We now describe how the functional \eqref{eq:Hpsi} is related to the large deviation functional associated to the 
Ginzburg-Landau dynamics. 
Let $\Omega_N = \{1,N\}$ be the one-dimensional periodic domain and $V$ a strictly convex potential (growing fast enough at infinity).
The Ginzburg-Landau dynamics is acting on the continuous variables $\xi = \{ \xi_i \}_{i \in \Omega_N}
\in \bbR^{\Omega_N }$
\begin{equation}
\label{SDE}
\dd \xi_i (t) = \sum_{j =i \pm 1} \left( V^\prime(\xi_j) - V^\prime(\xi_i) \right)dt + \sqrt{2}  \left(\dd B_{(i,i+1)}(t)- \dd B_{(i-1,i)}(t)\right).
\end{equation}
where $(B_{(i,i+1)} (t))_{i \in \gL}$ denote independent standard Brownian motions associated to each edge. We consider periodic boundary conditions for the moment.

The invariant measures are products and can be encoded by the density $f$ as for the ZRP (see section \ref{sec: The invariant measure})
\begin{equation}
\label{eq: invariant GL}
\mu_{f,N} (\dd \xi) = \bigotimes_{i \in \Omega_N} \mu_{f} \big( \dd \xi_i), 
\quad \text{with} \quad 
\mu_{f} \big( \dd \xi) = \frac{1}{Z_{\lambda(f)} } \exp \big( - V(\xi) + \lambda(f) \xi \big) \dd \xi.
\end{equation}
where the Lagrange parameter $\lambda(f)$  is tuned such that the mean density under $\mu_f$ is equal to $f$. Define also as in \eqref{eq: sigma micro}
\begin{equation}
\label{eq: sigma micro 2}
\gs ( f) = \bbE_{\mu_f} \big( V^\prime ( \xi) \big) = \gl(f) \, ,
\end{equation}
where the last equality follows by integration by parts.
It is shown in \cite{GPV} that after rescaling the Ginzburg-Landau dynamics follows the hydrodynamic equation $\partial_t f(x,t) = \Delta \gs \big( f (x,t) \big)$ (see \eqref{eq: hydro macro}). 

\medskip

Using similar considerations as in section \ref{sec: Static large deviations}, one can show that 
the large deviation function for the measure $\mu_{f_\infty,N}$ \eqref{eq: invariant GL} with constant density $f_\infty$ is given by 
\begin{eqnarray}
\label{eq: GD GL}
\cG(f) = \int_0^1 \left( \int_{f_\infty}^{f(x)} \left[ \sigma(s) -
    \sigma(f_\infty) \right] \dd s\right) \dd x. 
\end{eqnarray}
One recognizes the $\Psi$-entrophy \eqref{eq:Hpsi} with $\Psi(x) = x^2/2$.

\smallskip

Finally  the Ginzburg Landau dynamics \eqref{SDE} can be modified to take into account 
boundary terms
\begin{equation*}
\dd \xi_i (t)  = \sum_{j =i \pm 1} \left( V^\prime(\xi_j) -
  V^\prime(\xi_i) \right) \dd t + \sqrt{2}  \left(\dd B_{(i,i+1)}(t)- \dd
  B_{(i-1,i)}(t)\right)
\end{equation*}
for $i \not = 1, N$, and 
\begin{equation*}
\begin{cases} \displaystyle 
\dd \xi_1 (t)  = \left(V^\prime(\xi_2) - V^\prime(\xi_1)\right)\dd t + \sqrt{2} \dd B_{(1,2)}(t) 
+ (a - V^\prime(\xi_1) ) \dd t +  \dd B_0(t),\vspace{0.2cm} \\
\displaystyle 
\dd \xi_N (t)  = \left( V^\prime(\xi_{N-1}) - V^\prime(\xi_N) \right)
\dd t - \sqrt{2} \dd B_{(N-1,N)}(t) \vspace{0.2cm} \\ \displaystyle 
\qquad \qquad \qquad \qquad \qquad \qquad \qquad + \left(b -  V^\prime(\xi_1) \right) \dd t +  \dd
B_N(t),
\end{cases}
\end{equation*}
where $B_0$ and $B_N$ are two additional independent brownian motions acting at the boundaries.
The reservoirs impose the chemical potentials $a$ and $b$ at the boundaries.

As for the ZRP in contact with reservoirs, the invariant measure
remains a product $\otimes_{i=1}^N \mu^{\lambda_i}$ with a linearly
varying chemical potential $\lambda_i = a + (b-a) i/(N+1)$ (note that
$\mu^{\lambda_i}$ is the measure \eqref{eq: invariant GL} indexed by
the chemical potential instead of the density).  For this reason, the
large deviation functional \eqref{eq: GD GL} can be generalized when
the invariant density profile is no longer constant.

\subsection{Large deviations and Lyapunov functions}

In the previous section, we recalled that the ZRP in contact with
reservoirs satisfies a diffusion equation \eqref{eq: hydro macro
  dirichlet} and its invariant measure obeys a large deviation
principle \eqref{eq: LDP open}.  We will now show  that the large
deviation function is a Lyapunov function for the limiting equation of
the process.  As shown in \cite{Maes}, this statement holds in a very
general setting (see also \cite{BDGJL} for more analytic arguments).
Thus we recall the proof in the case of a general particle system.

\medskip

Consider a microscopic Markovian dynamics with stationary measure
$\nu_N$.  As in \eqref{eq: empirical}, $\dist (f,\eta)$ denotes a
distance for the weak topology between a density profile $f$ and the
empirical measure associated to $\eta$.

We suppose that the microscopic system approximates the solution of a
PDE (which is assumed to be unique).  We now quantify the convergence to
the hydrodynamic limit.
For any density profile $f$, define the conditional measure
\begin{eqnarray*}
\nu_{f,\gep,N} =  \nu_N \left( \cdot \;  \Big| \;  \dist(\eta,f) \leq \gep \right)  \,.
\end{eqnarray*}
The initial data will be drawn from the measure $\nu_{f_0,\gep,N}$ for $\gep$ small enough and therefore 
concentrate to $f_0$ (note that this initial measure differs slightly from the one in \eqref{eq: initial data}).
For any smooth  initial data $f_0$, we assume that the microscopic
dynamics remain close to the macroscopic profile $f_t$ at  any time $t >0 $ 
\begin{eqnarray}
\label{eq: assumption dynamique}
\forall \,\gd >0, \qquad 
\lim_{\gep \to 0}  \lim_{N \to \infty} \bbE_{\nu_{f_0 ,\gep,N}}  \left(  \dist(\eta(t) ,f_t ) \leq \gd \right) = 1  \,,
\end{eqnarray}
where $\bbE_{\nu_{f_0 ,\gep,N}}$ stands for the expectation wrt the dynamics starting from an initial data sampled from
$\nu_{f_0 ,\gep,N}$.  

Finally, we assume that
$\nu_N$ obeys a large deviation principle with functional $\cG$
\begin{eqnarray}
\label{eq: LD maths 2}
\cG(f) =
\lim_{\gep \to 0} \lim_{N \to \infty} \; -  \frac{1}{N^d} \ln \; \nu_N \Big(  \dist(\eta,f) \leq \gep  \Big)  \,.
\end{eqnarray}

For the ZRP, the previous assumptions are satisfied with the
hydrodynamic limit \eqref{eq: hydro macro dirichlet} and the large
deviation principle \eqref{eq: LDP open}.
\medskip

\begin{proposition}
  Under assumptions \eqref{eq: assumption dynamique} and \eqref{eq: LD
    maths 2}, the large deviation functional $\cG$ is a Lyapunov
  function
\begin{eqnarray}
\label{eq: lyapu 01}
\forall \, t \geq s, \quad  \cG(f_s) \geq \cG(f_t)  \,.
\end{eqnarray}
\end{proposition}

\begin{proof}
  Fix $\gd>0$. For $\gep >0$ small and   
$N$ large enough, assumption \eqref{eq: assumption dynamique} implies that
\begin{eqnarray*}
 \frac{1}{2} \nu_N \left(  \dist(\eta,f) \leq \gep \right) 
  \leq 
  \bbE_{\nu_N}  \left( \{ \dist(\eta(0) ,f ) \leq \gep \} \bigcap  \{ \dist(\eta(t) ,f_t ) \leq \gd \} \right) \,.
\end{eqnarray*}
Dropping the constraint on the initial data and using the fact that $\nu_N$
is the invariant measure
\begin{eqnarray*}
\frac{1}{2} \nu_N \left(  \dist(\eta,f) \leq \gep \right) 
\leq 
\bbE_{\nu_N}  \left(  \{ \dist(\eta(t) ,f_t ) \leq \gd \} \right) 
=  \nu_N \left(  \dist(\eta,f_t ) \leq \gd \right)  \,.
\end{eqnarray*}
From the large deviations \eqref{eq: LD maths 2}, we deduce that
$$
\forall \, t \geq 0, \qquad  \cG(f_0) \geq \cG(f_t) \, ,
$$
 by letting $\gep$ and $\gd$ go to 0.
 Repeating the argument at later times completes \eqref{eq: lyapu 01}.
\end{proof}

\medskip\par\noindent\emph{Acknowledgements.} This work started at a
meeting of the authors at IHES, where JLL was visiting E. Carlen and
M. Carvalho. CM also would like to thank J. Dolbeault for numerous
discussions on nonlinear diffusions, and for providing references on
this topic.  
We would also like to thank L. Bertini, D. Gabrielli and J. Lasinio for helpful comments.
The work of TB was supported by the ANR-2010-BLAN-0108.  The
work of JLL was supported by NSF grant DMR-1104501 and AFOSR grant
FA-9550-10-1-0131. The work of CM was supported by the ERC grant
MATKIT. \mk

\par\noindent{\scriptsize\copyright\ 2013 by the authors. This paper
  may be reproduced, in its entirety, for non-commercial purposes.}

\bibliographystyle{acm}
\bibliography{BLMV}

\signtb \signjl \signcm\signcv

\end{document}